\documentclass[3p,10pt,a4paper,twoside,fleqn,sort&compress]{filomat}
\usepackage{amssymb,amsmath,latexsym,mathrsfs}
\usepackage[varg]{pxfonts}

\newtheorem{theorem}{Theorem}[section]

\newtheorem{definition}[theorem]{Definition}

\newtheorem{lemma}[theorem]{Lemma}

\newcommand{\ds}{\displaystyle}
\newcommand{\FilVol}{xx}
\newcommand{\FilYear}{yyyy}
\newcommand{\FilPages}{zzz--zzz}
\newcommand{\FilDOI}{(will be added later)}
\newcommand{\FilReceived}{Received: dd Month yyyy; Accepted: dd Month yyyy}

\begin{document}

%

\title{Positive solutions for m-point p-Laplacian fractional boundary value problem involving Riemann Liouville fractional integral boundary conditions on the half line}

\author[affil1]{Dondu Oz}
\ead{dondu.oz@ege.edu.tr}
\author[affil2]{Ilkay Yaslan Karaca}
\ead{ilkay.karaca@ege.edu.tr}
\address[affil1]{Department of Mathematics, Ege University, 35100, Bornova, Izmir, Turkey}
\address[affil2]{Department of Mathematics, Ege University, 35100, Bornova, Izmir, Turkey}
\newcommand{\AuthorNames}{D. Oz, I.Y. Karaca}

\newcommand{\FilMSC}{Primary 34A08; Secondary 34B15, 26A33, 34B18, 47H10}
\newcommand{\FilKeywords}{Fractional calculus, Boundary value problem, P-Laplacian operator, Fixed point theorem, Positive solutions, Half line, Green function.}
\newcommand{\FilCommunicated}{Communicated by Maria Alessandra Ragusa}
\newcommand{\FilSupport}{The first author was granted a fellowship by the Scientific and Technological Research Council of Turkey
	(TUBITAK-2211-A).}

\begin{abstract}
This paper investigates the existence of positive solutions for m-point p-Laplacian fractional boundary value problem involving Riemann Liouville fractional integral boundary conditions on the half line via the Leray-Schauder Nonlinear Alternative theorem and the use and some properties of the Green function. As an application, an example is presented to demonstrate our main result.
\end{abstract}

\maketitle

\makeatletter
\renewcommand\@makefnmark%
{\mbox{\textsuperscript{\normalfont\@thefnmark)}}}
\makeatother

\section{Introduction}
Fractional differential equations arise in many engineering and scientific disciplines as the mathematical models of systems and processes in the fields of physics, mechanics, biology, chemistry, polymer rhcology, aero dynamics, capacitor theory, control theory, electrical circuits and other fields. There has been a noticable development in the study of fractional differential equations in recent years, see the monographs of Miller et al. \cite{book.5} and Agarwal et al. \cite{book.6} (see also \cite{article.7, article.15, article.16, article.17}).\\
\indent It should be noted that most of the papers on fractional calculus are devoted to the solvability of fractional differential equations on finite interval. Very recently, there are some papers concerning the fractional differential equations on infinite intervals, for example references \cite{article.1,article.5,article.9,article.13,article.18,article.19,article.22,article.23}.\\
\indent Recently, fractional differential equations with p-Laplacian operator have gained its importance and popularity due to its distinguished applications in numerous several fields of science and engineering, such as viscoelasticity mechanics, electrochemistry, fluid mechanics, non-Newtonian mechanics, combustion theory and material science. There have appeared some results for the existence of solutions or positive solutions of boundary value problems for fractional differential equations with p-Laplacian operator; see \cite{article.2, article.3, article.4, article.6,article.8,article.9,article.11,article.12,article.14,article.20,article.21,article.24,article.25,article.26,article.27} and the references therein. There is not work on positive solutions for m-point p-Laplacian fractional boundary value problem on the half line except that in \cite{article.19}.\\
\indent Liang et al.\cite{article.19} investigated the following m-point fractional boundary value problem with p-Laplacian on an infinite interval:
\begin{eqnarray*}
	\begin{cases}
		D_{0^+}^\gamma\left( \phi_{p}\left( D_{0^+}^\alpha u(t)\right) \right)  + a(t)f(t,u(t)) =0,\ 0<t< +\infty, \\
		u(0)=u'(0)=0, \ \ D_{0^+}^{\alpha-1} u(+\infty) = \displaystyle\sum_{i=1}^{m-2} \beta_i u(\xi_i),\ \ D_{0^+}^{\alpha} u(t)|_{t=0} =0,
	\end{cases}
\end{eqnarray*}
where $0<\gamma\leq1, \ 2<\alpha\leq3,\ D_{0^+}^\alpha$ is the standard Riemann-Liouville fractional derivative. $\phi_{p}(s)=|s|^{p-2}s,\ p>1,\ (\phi_{p})^{-1}=\phi_{q},\ \dfrac{1}{p}+\dfrac{1}{q}=1.\ 0<\xi_1<\xi_2<...<\xi_{m-2}<+\infty,\ \beta_i\geq0,\ i=1,2,...,m-2$ satisfies $0<\ds\sum_{i=1}^{m-2}\beta_i\xi_{i}^{\alpha-1}<\Gamma(\alpha),$ where $\Gamma(\alpha)$ is the Euler gamma function defined by $ \Gamma(\alpha)= \ds\int_{0}^{+\infty}t^{\alpha-1}e^{-t}dt, \ \ \alpha>0.$ They established solvability of the above fractional boundary value problems by means of the properties of the Green function and some fixed point theorems.\\
\indent Motivated the above paper, in this paper, we consider the following m-point p-Laplacian fractional boundary value problem (BVP) involving Riemann Liouville fractional integral boundary conditions on the half line: 
\begin{eqnarray}\label{eq1.1}
\begin{cases}
D_{0^+}^\gamma\left( \phi_{p}\left( D_{0^+}^\alpha u(t)\right) \right)  + a(t)f(t,u(t),u'(t)) =0,\ t\in[0, +\infty), \\
u(0)=u'(0)=0, \\ 
\displaystyle\lim_{t\rightarrow +\infty}D_{0^+}^{\alpha-1} u(t) = \displaystyle\sum_{i=1}^{m-2} \eta_i I_{0^+}^\beta u'(\xi_i),\ \ D_{0^+}^{\alpha} u(t)|_{t=0} =0,
\end{cases}
\end{eqnarray}\\
where $D_{0^+}^\gamma$ and $D_{0^+}^\alpha$ are the standard Riemann-Liouville fractional derivatives and $I_{0^+}^\beta$ is the standard Riemann-Liouville fractional integral with $0<\gamma\leq1, \ \ 2<\alpha\leq3$,\ \ $\beta>0,$\ $0<\xi_1<\xi_2<...<\xi_{m-2}< +\infty,$\ \ $i=1,...,m-2,$\ \ $\eta_i>0.$ The p-Laplacian operator is defined as $\phi_{p}(s)=|s|^{p-2}s$, $p>1$, $(\phi_{p})^{-1}=\phi_{q},$ $\dfrac{1}{p}+\dfrac{1}{q}=1.$ Throughout this paper we assume that following conditions hold:
\begin{itemize}
	\item[$(H1)$] $\eta_i>0$,\ \ $0< \displaystyle\sum_{i=1}^{m-2} \eta_i \xi_{i}^{\alpha+\beta-2}< \Gamma(\alpha+\beta-1);$
	\item[$(H2)$] $f\in \mathscr{C}([0,+\infty)\times[0,+\infty)\times[0,+\infty), [0,+\infty)),$ $f(t,0,0)\not\equiv 0$ on any subinterval of $(0,+\infty)$ and $f(t,(1+t^{\alpha-1})u,(1+t^{\alpha-1})v)$ is bounded when $u, \ v$ are bounded.
	\item[$(H3)$] $ a:[0, +\infty) \rightarrow [0, +\infty) $ is not identical zero on any closed subinterval of  $[0, +\infty) $  and $$\displaystyle\int_{0}^{+\infty}\phi_{q}\left( \dfrac{1}{\Gamma(\gamma)}\int_{0}^{s} (s-\tau)^{\gamma-1}a(\tau)d\tau\right)ds< +\infty.$$
\end{itemize}
\hspace{.5cm} By using Leray-Schauder Nonlinear Alternative theorem in \cite{book.6}, we get the existence of positive solutions for the BVP \eqref{eq1.1}. To the author's knowledge, the existence of positive solutions for m-point p-Laplacian fractional boundary value problems involving Riemann Liouville fractional integral boundary conditions on the half line is not investigated till now. Thus, this results can be considered as a contribution to this field. The organization of this paper is as follows. In section 2, we provide some definitions and preliminary lemmas which are key tools for our main result. In section 3, we give and prove our main result. Finally, in section 4, we give an example to illustrate how the main result can be used in practice.
\section{Preliminaries}
\hspace{.5cm} In this section, we introduce some preliminary facts which are used throughout this article. Now we recall the following definitions, which can be found in \cite{book.1,book.3,book.4,book.5}.
\begin{definition}\cite{book.1}\label{def2.1}
	The integral 
	\begin{equation*}
	I_{0^+}^{\alpha}f(t)=\ds\frac{1}{\Gamma(\alpha)}\ds\int_{0}^{t}(t-s)^{\alpha-1}f(s)ds,\ \ t>0, \ \ \alpha>0,
	\end{equation*}	
	is called {\it Riemann-Liouville} fractional integral of order $\alpha.$
\end{definition}
\begin{definition}\cite{book.1}\label{def2.2}
	For a function $f(t)$ given in the interval $[0, +\infty)$, the expression
	\begin{equation*}
	D_{0^+}^{\alpha}f(t)=\ds\frac{1}{\Gamma(n-\alpha)}\left( \ds\frac{d}{dt}\right) \ds\int_{0}^{t}\ds\frac{f(s)}{(t-s)^{\alpha-n+1}}ds,
	\end{equation*}	
	where $n=[\alpha]+1,$\ \ $[\alpha]$ denotes the integer part of number $\alpha$, is called the {\it Riemann-Liouville} fractional derivative of order $\alpha >0.$ 
\end{definition}
\begin{lemma}\cite{book.1}\label{lem2.3}
	Let $\alpha > 0.$ Assume that $u \in C(0,1)\cap L(0,1)$ with a fracdtional derivative of order $\alpha$ that belongs to $C(0,1)\cap L(0,1)$. Then 
	$$ I_{0^+}^{\alpha}D_{0^+}^{\alpha}u(t) = u(t)+c_{1}t^{\alpha - 1}+c_{2}t^{\alpha - 2}+c_{3}t^{\alpha - 3}+. . . +c_{n}t^{\alpha - n},$$ for some $c_{i} \in \mathbb{R},$\ \ $ i=1,2,...,n \ \ (n = [\alpha]+1)$.
\end{lemma}
\begin{lemma}\label{lem2.4}
	Let $\alpha, \beta >0.$ $f \in L^{1}[a,b].$ Then $I_{0^+}^{\alpha} I_{0^+}^{\beta} f(t)=I_{0^+}^{\alpha+\beta}f(t)=I_{0^+}^\beta I_{0^+}^\alpha f(t)$ and $D_{0^+}^\alpha I_{0^+}^\alpha f(t)=f(t)$,\ \ for all $t\in [a,b].$
\end{lemma}
\begin{lemma}\label{lem2.5}
	Let $\alpha, \beta >0$ and $n=[\alpha]+1,$ then the following relations hold:
	$$ D_{0^+}^\alpha t^{\beta-1}=\dfrac{\Gamma(\beta)}{\Gamma(\beta-\alpha)}t^{\beta-\alpha-1},\ \ \beta>n,$$
	$$ D_{0^+}^\alpha t^{k}=0, \ \ k=0,1,2,...,n-1.$$
\end{lemma}
To prove the main result of this paper we need the following lemma.
\begin{lemma}\label{lem2.6}
	Let $h \in C[0, +\infty)$ with $\ds\int_{0}^{+\infty}h(s)ds<\infty,$ then fractional BVP
	\begin{eqnarray}\label{eq2.1}
	\begin{cases}
	D_{0^+}^\alpha u(t)+h(t)=0, \ \ t\in \left[ 0, +\infty\right),\ \ 2<\alpha\leq3,\\
	u(0)=u'(0)=0,\\
	D_{0^+}^{\alpha-1} u(+\infty) = \displaystyle\sum_{i=1}^{m-2} \eta_i I_{0^+}^\beta u'(\xi_i)
	\end{cases} 
	\end{eqnarray}
	has a unique solution
	\begin{eqnarray}\label{eq2.2}
	u(t)=\displaystyle\int_{0}^{+\infty}G(t,s)h(s)ds,
	\end{eqnarray}
	where
	{\small\begin{eqnarray}\label{eq2.3}
		G(t,s)=
		\begin{cases}
		\ds\frac{[\Gamma(\alpha+\beta-1)-\ds\sum_{i=1}^{m-2}\eta_{i}(\xi_{i}-s)^{\alpha+\beta-2} ]t^{\alpha-1}-[\Gamma(\alpha+\beta-1)-\ds\sum_{i=1}^{m-2}\eta_{i}{\xi_{i}}^{\alpha+\beta-2} ](t-s)^{\alpha-1}}{\Gamma(\alpha)[\Gamma(\alpha+\beta-1)-\ds\sum_{i=1}^{m-2}\eta_{i}{\xi_{i}}^{\alpha+\beta-2}]},
		& s\leq \min\{t, \xi_{i}\},\\
		\ds\frac{[\Gamma(\alpha+\beta-1)-\ds\sum_{i=1}^{m-2}\eta_{i}(\xi_{i}-s)^{\alpha+\beta-2} ]t^{\alpha-1}}{\Gamma(\alpha)[\Gamma(\alpha+\beta-1)-\ds\sum_{i=1}^{m-2}\eta_{i}{\xi_{i}}^{\alpha+\beta-2}]},
		& 0\leq t\leq s\leq \xi_{i},\\
		\ds\frac{\Gamma(\alpha+\beta-1)t^{\alpha-1}-[\Gamma(\alpha+\beta-1)-\ds\sum_{i=1}^{m-2}\eta_{i}{\xi_{i}}^{\alpha+\beta-2} ](t-s)^{\alpha-1}}{\Gamma(\alpha)[\Gamma(\alpha+\beta-1)-\ds\sum_{i=1}^{m-2}\eta_{i}{\xi_{i}}^{\alpha+\beta-2}]},
		& 0\leq \xi_{i}\leq s\leq t,\\
		\ds\frac{\Gamma(\alpha+\beta-1)t^{\alpha-1}}{\Gamma(\alpha)[\Gamma(\alpha+\beta-1)-\ds\sum_{i=1}^{m-2}\eta_{i}{\xi_{i}}^{\alpha+\beta-2}]},
		& s\geq \max\{t, \xi_{i}\}.
		\end{cases}
		\end{eqnarray}}
	
\end{lemma}

\begin{proof}
	According to Lemma \ref{lem2.3}, the solution of \eqref{eq2.1} can be written as $$ u(t)=-I_{0^+}^\alpha h(t)+c_{1}t^{\alpha - 1}+c_{2}t^{\alpha - 2}+c_{3}t^{\alpha - 3}. $$
	By the boundary conditions of \eqref{eq2.1}, we know that $ c_{2}=0 $,  $ c_{3}=0 $.
	
	On the other hand, by $D_{0^+}^{\alpha-1} u(+\infty) = \displaystyle\sum_{i=1}^{m-2} \eta_i I_{0^+}^\beta u'(\xi_i)$, we have
	$$c_{1}=\ds\frac{\Gamma(\alpha+\beta-1)\ds\int_{0}^{+\infty}h(s)ds-\ds\sum_{i=1}^{m-2}\eta_{i}\ds\int_{0}^{\xi_i}(\xi_{i}-s)^{\alpha+\beta-2}h(s)ds}{\Gamma(\alpha)[\Gamma(\alpha+\beta-1)-\ds\sum_{i=1}^{m-2}\eta_{i}\xi_{i}^{\alpha+\beta-2}]}.$$ 
	Therefore, the unique solution of fractional BVP \eqref{eq2.1} is 
	\begin{eqnarray*}
		&u(t)&=-\ds\frac{1}{\Gamma(\alpha)}\ds\int_{0}^{t}(t-s)^{\alpha - 1}h(s)ds+\ds\frac{\Gamma(\alpha+\beta-1)t^{\alpha-1}}{\Gamma(\alpha)[\Gamma(\alpha+\beta-1)-\ds\sum_{i=1}^{m-2}\eta_{i}\xi_{i}^{\alpha+\beta-2}]}\ds\int_{0}^{+\infty}h(s)ds\\
		& \ \ & \ \ -\ds\frac{\ds\sum_{i=1}^{m-2}\eta_{i}t^{\alpha - 1}    }{\Gamma(\alpha)[\Gamma(\alpha+\beta-1)-\ds\sum_{i=1}^{m-2}\eta_{i}\xi_{i}^{\alpha+\beta-2}]}\ds\int_{0}^{\xi_{i}}(\xi_{i}-s)^{\alpha+\beta-2}h(s)ds\\
		& \ \ &
		=\displaystyle\int_{0}^{+\infty}G(t,s)h(s)ds,
	\end{eqnarray*}
	where $G(t,s)$ is defined by \eqref{eq2.3}.
\end{proof}

\begin{lemma}\label{lem2.7}
	BVP \eqref{eq1.1} is equivalent to the integral equation 
	\begin{equation}\label{eq2.4}
	u(t)=\displaystyle\int_{0}^{+\infty}G(t,s)\phi_{q}\left(\dfrac{1}{\Gamma(\gamma)}\ds\int_{0}^{s}(s-\tau)^{\gamma-1}a(\tau)f(\tau,u(\tau),u'(\tau))d\tau\right) ds,
	\end{equation}
	where $G(t,s)$ is defined by \eqref{eq2.3}.
\end{lemma}
\begin{proof}
	By the BVP \eqref{eq1.1} and Lemma \ref{lem2.3}, we have 
	$$ \phi_{p}\left( D_{0^+}^\alpha u(t) \right)=ct^{\gamma-1}-\dfrac{1}{\Gamma(\gamma)}\ds\int_{0}^{t}(t-s)^{\gamma-1}a(s)f(s,u(s),u'(s))ds. $$
	Together with $D_{0^+}^\alpha u(t)|_{t=0} = 0,$ there is $c=0,$ and then
	$$ D_{0^+}^\alpha u(t)=-\phi_{q}\left( \dfrac{1}{\Gamma(\gamma)}\ds\int_{0}^{t}(t-s)^{\gamma-1}a(s)f(s,u(s),u'(s))ds\right) . $$
	Therefore, BVP \eqref{eq1.1} is equivalent to the following problem
	\begin{eqnarray*}
		\begin{cases}
			D_{0^+}^\alpha u(t)+\phi_{q}\left( \dfrac{1}{\Gamma(\gamma)}\ds\int_{0}^{t}(t-s)^{\gamma-1}a(s)f(s,u(s),u'(s))ds\right)=0, \ \ t\in \left[ 0, +\infty\right),\\
			u(0)=u'(0)=0,\\
			D_{0^+}^{\alpha-1} u(+\infty) = \displaystyle\sum_{i=1}^{m-2} \eta_i I_{0^+}^\beta u'(\xi_i).
		\end{cases} 
	\end{eqnarray*}
	By Lemma \ref{lem2.6}, BVP \eqref{eq1.1} is equivalent to the integral equation \eqref{eq2.4}. The proof is complete.
\end{proof}

\begin{lemma}\label{lem2.8}
	If $(H1)$ holds, then for all $s,\ t\geq0$ we have
	\begin{eqnarray*}
		0\leq\displaystyle\frac {G(t,s)}{1+t^{\alpha-1}}\leq L, \ \ \ \ \ \ 0\leq\displaystyle\frac {G_{t}(t,s)}{1+t^{\alpha-1}}\leq (\alpha -1)L,
	\end{eqnarray*}\\
	where
	\begin{eqnarray}\label{eq.6}
	L=\displaystyle\frac{\Gamma(\alpha+\beta-1)}{\Gamma(\alpha)[  \Gamma(\alpha+\beta-1)-\displaystyle\sum_{i=1}^{m-2}\eta_i\xi_i^{\alpha+\beta-2}]}.
	\end{eqnarray}
\end{lemma}
\begin{proof}
	Simple computations give
	{\small \begin{eqnarray}\label{eq2.5}
		G_{t}(t,s)=
		\begin{cases}
		\ds\frac{[\Gamma(\alpha + \beta - 1)-\ds\sum_{i=1}^{m-2}\eta_{i}(\xi_{i}-s)^{\alpha + \beta - 2} ]t^{\alpha - 2}-[\Gamma(\alpha + \beta - 1)-\ds\sum_{i=1}^{m-2}\eta_{i}{\xi_{i}}^{\alpha + \beta - 2} ](t-s)^{\alpha - 2}}{\Gamma(\alpha - 1)[\Gamma(\alpha + \beta - 1)-\ds\sum_{i=1}^{m-2}\eta_{i}{\xi_{i}}^{\alpha + \beta - 2}]},
		& s\leq \min\{t, \xi_{i}\},\\
		\ds\frac{[\Gamma(\alpha + \beta - 1)-\ds\sum_{i=1}^{m-2}\eta_{i}(\xi_{i}-s)^{\alpha + \beta - 2} ]t^{\alpha-2}}{\Gamma(\alpha-1)[\Gamma(\alpha + \beta - 1)-\ds\sum_{i=1}^{m-2}\eta_{i}{\xi_{i}}^{\alpha + \beta - 2}]},
		& 0\leq t\leq s\leq \xi_{i},\\
		\ds\frac{\Gamma(\alpha + \beta - 1)t^{\alpha-2}-[\Gamma(\alpha + \beta - 1)-\ds\sum_{i=1}^{m-2}\eta_{i}{\xi_{i}}^{\alpha + \beta - 2} ](t-s)^{\alpha - 2}}{\Gamma(\alpha - 1)[\Gamma(\alpha + \beta - 1)-\ds\sum_{i=1}^{m-2}\eta_{i}{\xi_{i}}^{\alpha + \beta - 2}]},
		& 0\leq \xi_{i}\leq s\leq t,\\
		\ds\frac{\Gamma(\alpha + \beta - 1)t^{\alpha - 2}}{\Gamma(\alpha - 1)[\Gamma(\alpha + \beta - 1)-\ds\sum_{i=1}^{m-2}\eta_{i}{\xi_{i}}^{\alpha + \beta - 2}]},
		& s\geq \max\{t, \xi_{i}\}.
		\end{cases}
		\end{eqnarray}}
	Let us consider the case $s\leq \min\{t, \xi_{i}\},$ then we get
	
	\begin{eqnarray*}
	G(t,s)&=&\ds\frac{[\Gamma(\alpha+\beta-1)-\ds\sum_{i=1}^{m-2}\eta_{i}(\xi_{i}-s)^{\alpha+\beta-2} ]t^{\alpha-1}-[\Gamma(\alpha+\beta-1)-\ds\sum_{i=1}^{m-2}\eta_{i}{\xi_{i}}^{\alpha+\beta-2} ](t-s)^{\alpha-1}}{\Gamma(\alpha)[\Gamma(\alpha+\beta-1)-\ds\sum_{i=1}^{m-2}\eta_{i}{\xi_{i}}^{\alpha+\beta-2}]}\\ \\ & \ \ &
	\geq \ds\frac{[\Gamma(\alpha+\beta-1)-\ds\sum_{i=1}^{m-2}\eta_{i}{\xi_{i}}^{\alpha+\beta-2}](t^{\alpha-1}-(t-s)^{\alpha-1})}{\Gamma(\alpha)[\Gamma(\alpha+\beta-1)-\ds\sum_{i=1}^{m-2}\eta_{i}{\xi_{i}}^{\alpha+\beta-2}]}\\ \\ & \ \ &
	=\ds\frac{t^{\alpha-1}-(t-s)^{\alpha-1}}{\Gamma(\alpha)}\\ \\ & \ \ & \geq 0,
\end{eqnarray*}
	\begin{eqnarray*}
	G_{t}(t,s)&=&\ds\frac{[\Gamma(\alpha+\beta-1)-\ds\sum_{i=1}^{m-2}\eta_{i}(\xi_{i}-s)^{\alpha+\beta-2} ]t^{\alpha-2}-[\Gamma(\alpha+\beta-1)-\ds\sum_{i=1}^{m-2}\eta_{i}{\xi_{i}}^{\alpha+\beta-2} ](t-s)^{\alpha-2}}{\Gamma(\alpha-1)[\Gamma(\alpha+\beta-1)-\ds\sum_{i=1}^{m-2}\eta_{i}{\xi_{i}}^{\alpha+\beta-2}]}
\end{eqnarray*}

\begin{eqnarray*}
& \ \ &
	\geq \ds\frac{[\Gamma(\alpha+\beta-1)-\ds\sum_{i=1}^{m-2}\eta_{i}{\xi_{i}}^{\alpha+\beta-2}](t^{\alpha-2}-(t-s)^{\alpha-2})}{\Gamma(\alpha-1)[\Gamma(\alpha+\beta-1)-\ds\sum_{i=1}^{m-2}\eta_{i}{\xi_{i}}^{\alpha+\beta-2}]}\\ \\ & \ \ &
	=\ds\frac{t^{\alpha-2}-(t-s)^{\alpha-2}}{\Gamma(\alpha-1)}\\ \\ & \ \ & \geq 0.
\end{eqnarray*}	
If $ s\leq \min\{t, \xi_{i}\}$ then
\begin{eqnarray*}
	&\ds\frac{G(t,s)}{1+t^{\alpha-1}} & \leq \ds\frac{\Gamma(\alpha+\beta-1)(t^{\alpha-1}-(t-s)^{\alpha-1})+\ds\sum_{i=1}^{m-2}\eta_{i}{\xi_{i}}^{\alpha+\beta-2}(t-s)^{\alpha-1} }{(1+t^{\alpha-1})\Gamma(\alpha)[\Gamma(\alpha+\beta-1)-\ds\sum_{i=1}^{m-2}\eta_{i}{\xi_{i}}^{\alpha+\beta-2}]}\\ \\ & \ \ &
	\leq \ds\frac{\Gamma(\alpha+\beta-1)\ds\frac{t^{\alpha-1}}{1+t^{\alpha-1}}+[-\Gamma(\alpha+\beta-1)+\ds\sum_{i=1}^{m-2}\eta_{i}{\xi_{i}}^{\alpha+\beta-2}]\ds\frac{(t-s)^{\alpha-1}}{1+t^{\alpha-1}}}{\Gamma(\alpha)[\Gamma(\alpha+\beta-1)-\ds\sum_{i=1}^{m-2}\eta_{i}{\xi_{i}}^{\alpha+\beta-2}]}\\ \\
	& \ \ & 	
	\leq \ds\frac{\Gamma(\alpha+\beta-1)}{\Gamma(\alpha)[\Gamma(\alpha+\beta-1)-\ds\sum_{i=1}^{m-2}\eta_{i}{\xi_{i}}^{\alpha+\beta-2}]}-\ds\frac{\ds\frac{(t-s)^{\alpha-1}}{1+t^{\alpha-1}}}{\Gamma(\alpha)}\\ \\ & \ \ & 	
	\leq
	\ds\frac{\Gamma(\alpha+\beta-1)}{\Gamma(\alpha)[\Gamma(\alpha+\beta-1)-\ds\sum_{i=1}^{m-2}\eta_{i}{\xi_{i}}^{\alpha+\beta-2}]}\\ \\ & \ \ &
	=L,
\end{eqnarray*}
\begin{eqnarray*}
	&\ds\frac{G_{t}(t,s)}{1+t^{\alpha-1}}&\leq \ds\frac{\Gamma(\alpha+\beta-1)(t^{\alpha-2}-(t-s)^{\alpha-2})+\ds\sum_{i=1}^{m-2}\eta_{i}{\xi_{i}}^{\alpha+\beta-2}(t-s)^{\alpha-2 }}{(1+t^{\alpha-1})\Gamma(\alpha-1)[\Gamma(\alpha+\beta-1)-\ds\sum_{i=1}^{m-2}\eta_{i}{\xi_{i}}^{\alpha+\beta-2}]}	\\ \\ & \ \ &
	\leq \ds\frac{\Gamma(\alpha+\beta-1)\ds\frac{t^{\alpha-2}}{1+t^{\alpha-1}}+[-\Gamma(\alpha+\beta-1)+\ds\sum_{i=1}^{m-2}\eta_{i}{\xi_{i}}^{\alpha+\beta-2}]\ds\frac{(t-s)^{\alpha-2}}{1+t^{\alpha-1}}}{\Gamma(\alpha-1)[\Gamma(\alpha+\beta-1)-\ds\sum_{i=1}^{m-2}\eta_{i}{\xi_{i}}^{\alpha+\beta-2}]}
\end{eqnarray*}
\begin{eqnarray*}	
\\ \\	
	& \ \ & \leq\ds\frac{\Gamma(\alpha+\beta-1)}{\Gamma(\alpha-1)[\Gamma(\alpha+\beta-1)-\ds\sum_{i=1}^{m-2}\eta_{i}{\xi_{i}}^{\alpha+\beta-2}]}-\ds\frac{\ds\frac{(t-s)^{\alpha-2}}{1+t^{\alpha-1}}}{\Gamma(\alpha-1)}\\ \\ & \ \ & 	
	\leq\ds\frac{\Gamma(\alpha+\beta-1)}{\Gamma(\alpha-1)[\Gamma(\alpha+\beta-1)-\ds\sum_{i=1}^{m-2}\eta_{i}{\xi_{i}}^{\alpha+\beta-2}]}\\ \\ & \ \ & = (\alpha-1)L.
\end{eqnarray*}
Applying the same techniques to the other cases, the conclusion follows.
\end{proof}	
\indent In this paper, we will use the Banach space $E$ defined by	
\begin{eqnarray*}
E= \left\lbrace u\in \mathscr{C^{1}}(\mathbb{R^+},\mathbb{R^+}) : 	\ds\lim_{t \rightarrow \infty} \displaystyle\frac{|u(t)|}{1+t^{\alpha-1}}< \infty, \ds\lim_{t \rightarrow \infty} \displaystyle\frac{|u'(t)|}{1+t^{\alpha-1}}< \infty\right\rbrace.
\end{eqnarray*}
are equipped with the norm $\|u\|=\ds\max\left\lbrace \|u\|_\infty, \|u'\|_\infty \right\rbrace $, where $\|u\|_\infty = \ds\sup_{t\geq 0}\displaystyle\frac{|u(t)|}{1+t^{\alpha-1}}$ and $\mathbb{R^+}=[0, \infty).$ Applying some standard arguments about properties of a given Banach space, we can show that $E$ is Banach spaces. Basically in this paper, we use the Banach space $E$ defined above.
We introduce an operator $T: E\rightarrow E$ as follows 
\begin{eqnarray}\label{eq2.6}
(Tu)(t)=\displaystyle\int_{0}^{+\infty}G(t,s)\phi_{q}\left(\dfrac{1}{\Gamma(\gamma)}\ds\int_{0}^{s}(s-\tau)^{\gamma-1}a(\tau)f(\tau,u(\tau),u'(\tau))d\tau\right)ds,
\end{eqnarray}
where $G(t,s)$ is defined by \eqref{eq2.3}.\\
\indent It can be said that $u$ is a solution of the fractional BVP \eqref{eq1.1} if and only if $u$ is a
fixed point of the operator $T$ on $E.$ The following fixed point theorem is fundamental and essential to the proofs of
our main results.
\begin{theorem}[Leray-Schauder Nonlinear Alternative Theorem]\label{teo2.9}\cite{book.6}
	Let $ C $ be a convex subset of a Banach space, $ U $ be a open subset of $C$ with $0 \in U$. Then every completely continuous map $ T: \bar{U}\rightarrow C $ has at least one of the two following properties:\\
	\indent\textbf{$ (E_1) $} There exist an $ u\in \bar{U} $ such that $ Tu=u $.\\
	\indent\textbf{$ (E_2) $} There exist an $ u\in \partial{U} $ and $\lambda\in(0,1)$ such that $ u=\lambda Tu.$
\end{theorem}
As a result of noncompactness of half line $[0,\infty)$, the $ Arzela-Ascoli $ theorem fails to work in space $ E. $ Thus in order to show the compactness of the operator $ T $ defined by \eqref{eq2.6}, we need to represent to following modified compactness criterion.
\begin{lemma}\label{lem2.10}\cite{book.2}
	Let $ V=\left\lbrace u\in C_{\infty}, \|u\|<l, where\ \ l>0\right\rbrace,$ $V(t)=\{\dfrac{u(t)}{1+t^{\alpha-1}}, u\in V\}, $ $V'(t)=\{\dfrac{u'(t)}{1+t^{\alpha-1}}, u\in V\}.$ $V$ is relatively compact in $E,$ if \ $ V(t) $ and $ V'(t) $ are both equicontinuous on any finite subinterval of $\mathbb{R^+}$ and equiconvergent at $\infty,$ that is  for any $ \epsilon>0 $, there exists $ \eta=\eta(\epsilon)>0 $ such that
	\begin{eqnarray*}
		\left|\ds\frac{u(t_{1})}{1+t_{1}^{\alpha-1}}-\ds\frac{u(t_{2})}{1+t_{2}^{\alpha-1}} \right|<\epsilon, \ \  \left|\ds\frac{u'(t_{1})}{1+t_{1}^{\alpha-1}}-\ds\frac{u'(t_{2})}{1+t_{2}^{\alpha-1}} \right|<\epsilon,
	\end{eqnarray*}
	$ \forall u\in V,\ \ t_{1}, t_{2}\geq \eta $ (uniformly according to u).
\end{lemma}
\begin{lemma}\label{lem2.11}
	If conditions $(H1)$-$(H3)$ hold, then the operator $T: E\rightarrow E$ is completely continuous.
\end{lemma}
\begin{proof}
	\indent In order to represent the proof, we divide it into the three steps as follows:\\
	\text{{Step 1}:} In this step show that integral operator $T: E\rightarrow E$ is continuous. Assume that $u_{n}$ be a sequence in $E$
	such that $u_n\rightarrow u$ and $u'_n\rightarrow u'$ as $n \rightarrow +\infty$. Thus there exist positive constant $ r_0 $ such that
	$$\ds\max \left\lbrace \left\|u\right\|_{\infty}, \ds\sup_{n\in \mathbb{N}}\left\|u_{n}\right\|_{\infty}\right\rbrace, \ds\max \left\lbrace \left\|u'\right\|_{\infty}, \ds\sup_{n\in \mathbb{N}}\left\|u'_{n}\right\|_{\infty}\right\rbrace< r_0.$$ 
	With the help of $ Lebesgue Dominated Convergence$ theorem and continuity of f, we conclude that 
	\begin{eqnarray*}
		&\displaystyle\int_{0}^{+\infty}&\phi_{q}\left(\dfrac{1}{\Gamma(\gamma)}\ds\int_{0}^{s}(s-\tau)^{\gamma-1}a(\tau)f(\tau,u_{n}(\tau),u_{n}'(\tau))d\tau\right)ds\\ \\ & \ \ & 
		\rightarrow \displaystyle\int_{0}^{+\infty}\phi_{q}\left(\dfrac{1}{\Gamma(\gamma)}\ds\int_{0}^{s}(s-\tau)^{\gamma-1}a(\tau)f(\tau,u(\tau),u'(\tau))d\tau\right)ds, \ \ n\rightarrow\infty.
	\end{eqnarray*}
	Therefore considering Lemma \ref{lem2.8}, we can get
	\begin{eqnarray*}
		\left\|Tu_{n}-Tu\right\|_{\infty}&
		\leq& L\bigg| \displaystyle\int_{0}^{+\infty}\phi_{q}\left(\dfrac{1}{\Gamma(\gamma)}\ds\int_{0}^{s}(s-\tau)^{\gamma-1}a(\tau)f(\tau,u_{n}(\tau),u_{n}'(\tau))d\tau\right)ds\\ \\& \ \ &
		-\displaystyle\int_{0}^{+\infty}\phi_{q}\left(\dfrac{1}{\Gamma(\gamma)}\ds\int_{0}^{s}(s-\tau)^{\gamma-1}a(\tau)f(\tau,u(\tau),u'(\tau))d\tau\right)ds\bigg|\\ \\ & \ \ &
		\rightarrow 0, \ \ n\rightarrow+\infty,
	\end{eqnarray*}
	\begin{eqnarray*}
		\left\|T'u_{n}-T'u\right\|_{\infty}&
		\leq& (\alpha-1)L\bigg| \displaystyle\int_{0}^{+\infty}\phi_{q}\left(\dfrac{1}{\Gamma(\gamma)}\ds\int_{0}^{s}(s-\tau)^{\gamma-1}a(\tau)f(\tau,u_{n}(\tau),u_{n}'(\tau))d\tau\right)ds\\ \\& \ \ &
		-\displaystyle\int_{0}^{+\infty}\phi_{q}\left(\dfrac{1}{\Gamma(\gamma)}\ds\int_{0}^{s}(s-\tau)^{\gamma-1}a(\tau)f(\tau,u(\tau),u'(\tau))d\tau\right)ds\bigg|\\ \\ & \ \ &
		\rightarrow 0, \ \ n\rightarrow+\infty.
	\end{eqnarray*}
	Therefore 
	\begin{eqnarray*}
		\left\|Tu_{n}-Tu\right\|\rightarrow 0, \ \ n\rightarrow+\infty.
	\end{eqnarray*}
	Thus, $T$ is continuous. \\
	\text{{Step 2}:} In order to prove the relatively compactness of operator $T: E\rightarrow E.$ From the definition of $E$, we can choose $r_0$ such that $\ds\sup_{u\in E}\|u\| <r_0.$ Let $$B_{r_0}=\ds\sup\{f(t,(1+t^{\alpha-1})u,(1+t^{\alpha-1})u'), \ \ (t,u,u')\in[0,+\infty)\times[0,r_0]\times[0,r_0]\}$$
	and $\Omega$ be any bounded subset of $E.$ Then there exists $r>0$ such that  $\left\|u\right\|\leq r$  for all $u\in\Omega$. Then using conditions $(H2), (H3)$ and Lemma \ref{lem2.8}, we have 
	\begin{eqnarray*}
		&\left\|Tu\right\|_{\infty}&\leq L\ds\int_{0}^{+\infty}\phi_{q}\left(\dfrac{1}{\Gamma(\gamma)}\ds\int_{0}^{s}(s-\tau)^{\gamma-1}a(\tau)f(\tau,u(\tau),u'(\tau))d\tau\right)ds\\ \\ & \ \ &
		\leq L\phi_{q}(B_r)\ds\int_{0}^{+\infty}\phi_{q}\left(\dfrac{1}{\Gamma(\gamma)}\ds\int_{0}^{s}(s-\tau)^{\gamma-1}a(\tau)d\tau\right) ds\\ \\
		& \ \
		&< +\infty,\ \ u\in \Omega.
	\end{eqnarray*}
	Similarly we can show that $\left\|(Tu)'\right\|_{\infty}<\infty $ for $ u\in \Omega. $ It show that $ T\Omega $ is uniformly bounded. Next, we show that $T\Omega$ is equicontinuous on $[0, +\infty)$. For any $a > 0$ and $t_1,\ t_2\in [0, a],$ without loss of generality, we may assume that $t_2 > t_1$. For all $ u\in \Omega$, we have
	
	\begin{eqnarray*}
		&\left|\ds\frac{(Tu)(t_2)}{1+t_{2}^{\alpha-1}}-\ds\frac{(Tu)(t_1)}{1+t_{1}^{\alpha-1}}\right| &\leq \ds\int_{0}^{+\infty}\left| \ds\frac{G(t_{2},s)}{1+t_{2}^{\alpha-1}}-\ds\frac{G(t_{1},s)}{1+t_{1}^{\alpha-1}}\right|\phi_{q}\left(\dfrac{1}{\Gamma(\gamma)}\ds\int_{0}^{s}(s-\tau)^{\gamma-1}a(\tau)f(\tau,u(\tau),u'(\tau))d\tau\right)ds\\ \\
		& \ \
		&\leq \ds\int_{0}^{+\infty}\left| \ds\frac{G(t_{2},s)}{1+t_{2}^{\alpha-1}}-\ds\frac{G(t_{1},s)}{1+t_{2}^{\alpha-1}}\right|\phi_{q}\left(\dfrac{1}{\Gamma(\gamma)}\ds\int_{0}^{s}(s-\tau)^{\gamma-1}a(\tau)f(\tau,u(\tau),u'(\tau))d\tau\right)ds\\ \\
		& \ \
		& \ \ +\ds\int_{0}^{+\infty}\left| \ds\frac{G(t_{1},s)}{1+t_{2}^{\alpha-1}}-\ds\frac{G(t_{1},s)}{1+t_{2}^{\alpha-1}}\right|\phi_{q}\left(\dfrac{1}{\Gamma(\gamma)}\ds\int_{0}^{s}(s-\tau)^{\gamma-1}a(\tau)f(\tau,u(\tau),u'(\tau))d\tau\right)ds\\ \\
		& \ \
		&\leq \ds\int_{0}^{+\infty}\left| \ds\frac{G(t_{2},s)-G(t_{1},s)}{1+t_{2}^{\alpha-1}}\right|\phi_{q}\left(\dfrac{1}{\Gamma(\gamma)}\ds\int_{0}^{s}(s-\tau)^{\gamma-1}a(\tau)f(\tau,u(\tau),u'(\tau))d\tau\right)ds\\ \\
		& \ \
		& \ \ +\ds\int_{0}^{+\infty} \ds\frac{G(t_{1},s)|t_{2}^{\alpha-1}-t_{1}^{\alpha-1}|}{(1+t_{1}^{\alpha-1})(1+t_{2}^{\alpha-1})}\phi_{q}\left(\dfrac{1}{\Gamma(\gamma)}\ds\int_{0}^{s}(s-\tau)^{\gamma-1}a(\tau)f(\tau,u(\tau),u'(\tau))d\tau\right)ds\\ \\
		& \ \
		&\leq \ds\frac{2L|t_{2}^{\alpha-1}-t_{1}^{\alpha-1}|}{1+t_{2}^{\alpha-1}}\ds\int_{0}^{+\infty}\phi_{q}\left(\dfrac{1}{\Gamma(\gamma)}\ds\int_{0}^{s}(s-\tau)^{\gamma-1}a(\tau)f(\tau,u(\tau),u'(\tau))d\tau\right)ds
	\end{eqnarray*}
	So we conclude that  
	\begin{eqnarray*}
		&\left|\ds\frac{Tu(t_2)}{1+t_{2}^{\alpha-1}}-\ds\frac{Tu(t_1)}{1+t_{1}^{\alpha-1}}\right|& \leq \ds\frac{2L|t_{2}^{\alpha-1}-t_{1}^{\alpha-1}|}{1+t_{2}^{\alpha-1}}\ds\int_{0}^{+\infty}\phi_{q}\left(\dfrac{1}{\Gamma(\gamma)}\ds\int_{0}^{s}(s-\tau)^{\gamma-1}a(\tau)f(\tau,u(\tau),u'(\tau))d\tau\right)ds\\ \\ & \ \ &\rightarrow 0 \ \ \emph{as uniformly} \ \ t_1 \rightarrow t_2 \ \ for \ \ u\in\Omega.
	\end{eqnarray*}
	Similarly we can prove that
	$$\left|\ds\frac{T'u(t_2)}{1+t_{2}^{\alpha-1}}-\ds\frac{T'u(t_1)}{1+t_{1}^{\alpha-1}}\right|\rightarrow 0,$$
	when uniformly $t_1 \rightarrow t_2.$ Hence, $T\Omega$ is equicontinuous on $[0, +\infty).$\\
	\textbf{{Step 3}:} At last we must prove that $T\Omega$ is equiconvergent at infinity.
	For any $u\in \Omega$, we have
	\begin{eqnarray*}
		\ds\int_{0}^{+\infty}\phi_{q}\left(\dfrac{1}{\Gamma(\gamma)}\ds\int_{0}^{s}(s-\tau)^{\gamma-1}a(\tau)f(\tau,u(\tau),u'(\tau))d\tau\right)ds\leq \phi_{q}(B_r)\ds\int_{0}^{+\infty}\phi_{q}\left(\dfrac{1}{\Gamma(\gamma)}\ds\int_{0}^{s}(s-\tau)^{\gamma-1}a(\tau)d\tau\right)ds < +\infty.
	\end{eqnarray*}
	From Lemma \ref{lem2.8}, we can get
	\begin{eqnarray*}
		&\ds\lim_{t\rightarrow +\infty}\left| \ds\frac{Tu(t)}{1+t^{\alpha-1}}\right|&=\lim_{t\rightarrow +\infty}\left|\ds\int_{0}^{+\infty}\ds\frac{G(t,s)}{1+t^{\alpha-1}}\phi_{q}\left(\dfrac{1}{\Gamma(\gamma)}\ds\int_{0}^{s}(s-\tau)^{\gamma-1}a(\tau)f(\tau,u(\tau),u'(\tau))d\tau\right)ds\right|\\ \\
		& \ \ 
		&\leq L\ds\int_{0}^{+\infty}\phi_{q}\left(\dfrac{1}{\Gamma(\gamma)}\ds\int_{0}^{s}(s-\tau)^{\gamma-1}a(\tau)f(\tau,u(\tau),u'(\tau))d\tau\right)ds\\ \\
		& \ \
		&< +\infty.
	\end{eqnarray*}
	Similarly we can obtain the following
	$$\ds\lim_{t\rightarrow +\infty}\left| \ds\frac{(Tu)'(t)}{1+t^{\alpha-1}}\right|< +\infty.$$
	Hence, $T\Omega$ is equiconvergent at infinity. Consequently, by means of compactness criterion in Lemma \ref{lem2.10}, we deduce that integral operator $T: E\rightarrow E$ is completely continuous operator.
\end{proof}
\indent For easy statement, denote
\begin{eqnarray}\label{eq.9}
M=L(\alpha-1)\phi_{q}(B_\delta)\ds\int_{0}^{+\infty}\phi_{q}\left(\dfrac{1}{\Gamma(\gamma)}\ds\int_{0}^{s}(s-\tau)^{\gamma-1}a(\tau)d\tau\right)ds 
\end{eqnarray}
where \textit{L} is defined by \eqref{eq.6}.
\section{Main result}
\begin{theorem}\label{teo3.1}
	Let that conditions $(H1)-(H3)$ hold and the following condition is satisfied:\\ there exist positive constant $ \delta $ such that 
	\begin{eqnarray}\label{eq3.1}
	\ds\frac{\delta}{M}\geq 1.
	\end{eqnarray}
	Then the fractional BVP \eqref{eq1.1} has a positive solution $ u=u(t) $ such that 
	$$ 0\leq\ds\frac{u(t)}{1+t^{\alpha-1}}\leq\delta, \ \  0\leq\ds\frac{u'(t)}{1+t^{\alpha-1}}\leq\delta, \ \ t\in[0, +\infty).$$
\end{theorem}
\begin{proof}
	Let us consider the following fractional BVP
	\begin{eqnarray}\label{eq3.2}
	\begin{cases}
	D_{0^+}^\gamma\left( \phi_{p}\left( D_{0^+}^\alpha u(t)\right) \right)  + \lambda(t)f(t,u(t),u'(t)) =0,\ \ t\in[0, +\infty),\ \ 2<\alpha\leq3, \ \ \lambda\in (0,1),\\
	u(0)=u'(0)=0, \\ 
	\displaystyle\lim_{t\rightarrow +\infty}D_{0^+}^{\alpha-1} u(t) = \displaystyle\sum_{i=1}^{m-2} \eta_i I_{0^+}^\beta u'(\xi_i),\ \ D_{0^+}^{\alpha} u(t)|_{t=0} =0,\\
	0<\xi_1<\xi_2<...<\xi_{m-2}< +\infty, \ \ \eta_{i}>0, \ \ \beta>0, \ \ 0<\gamma\leq1.
	\end{cases}
	\end{eqnarray}
	We know that solving \eqref{eq3.2} is equivalent to solving the fixed point problem $ u=\lambda Tu. $\\
	Assume that $$U=\left\lbrace u\in E\ \ | \ \ \|u\|<\delta\right\rbrace.$$
	We claim that there is no $u\in\partial U$ such that $u=\lambda Tu$ for  $\lambda\in(0, 1).$\\
	The proof is immediate, because if there exist $ u\in\partial U $ with $ u=\lambda Tu,$ then for $\lambda \in(0, 1)$ we have 
	\begin{eqnarray*}
		&\|u(t)\|_{\infty}& = \|\lambda(Tu)(t)\|_{\infty} = \ds\sup_{t\in [0,+\infty)}\lambda\left| \dfrac{(Tu)(t)}{1+t^{\alpha-1}}\right|\\ \\ & \ \ & 
		< \ds\sup_{t\in [0,+\infty)}\ds\int_{0}^{+\infty} \ds\frac{G(t,s)}{1+t^{\alpha-1}}\phi_{q}\left(\dfrac{1}{\Gamma(\gamma)}\ds\int_{0}^{s}(s-\tau)^{\gamma-1}a(\tau)f(\tau,u(\tau),u'(\tau))d\tau\right)ds\\ \\ & \ \ &
		\leq L\ds\int_{0}^{+\infty}\phi_{q}\left(\dfrac{1}{\Gamma(\gamma)}\ds\int_{0}^{s}(s-\tau)^{\gamma-1}a(\tau)f( \tau,(1+\tau^{\alpha-1})\ds\frac{u(\tau)}{1+\tau^{\alpha-1}},(1+\tau^{\alpha-1})\ds\frac{u'(\tau)}{1+\tau^{\alpha-1}})\right) ds\\ \\ & \ \ & 
		\leq L\phi_{q}(B_\delta)\ds\int_{0}^{+\infty}\phi_{q}\left(\dfrac{1}{\Gamma(\gamma)}\ds\int_{0}^{s}(s-\tau)^{\gamma-1}a(\tau)d\tau\right)ds\\ \\ & \ \ & 
		\leq M.
	\end{eqnarray*}
	Analogously we can show that 
	\begin{eqnarray*}
		&\|u'(t)\|_{\infty}&=\|\lambda(Tu)'(t)\|_{\infty}=\ds\sup_{t\in [0,+\infty)}\lambda\left| \dfrac{(Tu)'(t)}{1+t^{\alpha-1}}\right|< M.
	\end{eqnarray*}
	Both inequality, we conclude that $\|u\|=\|\lambda Tu\|<M.$ This yields that
	\begin{eqnarray*}
		\ds\frac{\delta}{M}<1,
	\end{eqnarray*}
	which is contradiction with \eqref{eq3.1}. Then by means of Theorem \ref{teo3.1}, the fractional BVP \eqref{eq1.1} has a positive solution $u = u(t)$ such that
	$$ 0\leq\ds\frac{u(t)}{1+t^{\alpha-1}}\leq\delta, \ \  0\leq\ds\frac{u'(t)}{1+t^{\alpha-1}}\leq\delta, \ \ t\in[0, +\infty).$$
	This completes the proof.
\end{proof}

\section{\large\bf Example}

\textbf{Example 4.1}
	Let $ \alpha = \dfrac{5}{2},\ \ \beta = \displaystyle\frac{1}{2},\ \ \eta_1 = \eta_2 = \xi_1 = \displaystyle\frac{1}{3}, \ \ \xi_2 = \displaystyle\frac{2}{3}, \ \ m = 4, \ \ p = 2, \ \ \phi_{p}(u) = u,\ \ \gamma = 1, \\ f(t, u, v) = \dfrac{1}{9}\left( \dfrac{2u}{1+t^\frac{3}{2}}+\dfrac{u'}{1+t^\frac{3}{2}}+1\right) $ in the BVP \eqref{eq1.1}. Now we consider the following fractional BVP
	\begin{eqnarray}\label{eq4.1}
	\begin{cases}
	D_{0^+}^{1}\left( \phi_{p}\left( D_{0^+}^{\frac{5}{2}} u(t)\right) \right)  + a(t)f(t,u(t),u'(t)) =0,\ \ t\in[0, +\infty), \\
	u(0)=u'(0)=0, \\ 
	\displaystyle\lim_{t\rightarrow +\infty}D_{0^+}^{\frac{3}{2}} u(t) = \displaystyle\sum_{i=1}^{2} \eta_i I_{0^+}^{\frac{1}{2}} u'(\xi_i),\ \ D_{0^+}^{\frac{5}{2}} u(t)|_{t=0}=0.
	\end{cases}
	\end{eqnarray}
	Now, we show that the conditions (H1)-(H3) hold. $\eta_1=\eta_2=\dfrac{1}{3}>0$, $\ds\sum_{i=1}^{2}\eta_{i}\xi_{i}=\dfrac{1}{3}<\Gamma(2)=1.$ Then the condition (H1) is satisfied. $f(t,0,0)=\dfrac{1}{9}\not\equiv 0$ on any subinterval of $(0,+\infty)$. Let $u, v$ are bounded, then there exists  $\delta>0 $ such that $\|u\|_\infty\leq \delta$ and $\ \|v\|_\infty\leq \delta \ .$ We have
	$B_{\delta}=\ds\sup\{f(t, (1+t^\frac{3}{2})u, (1+t^\frac{3}{2})v), \ \ (t,u,v)\in[0,+\infty)\times[0,\delta]\times[0,\delta]\}=\dfrac{3\delta+1}{9}$ since $ f(t, (1+t^\frac{3}{2})u, (1+t^\frac{3}{2})v) = \dfrac{2u+v+1}{9}\leq \dfrac{3\delta+1}{9}$ for $ (t,u,v)\in[0,+\infty)\times[0,\delta]\times[0,\delta]. $ This yields that $f(t, (1+t^\frac{3}{2})u, (1+t^\frac{3}{2})v)$ is bounded. Thus the condition (H2) is satisfied. We take $$\displaystyle\int_{0}^{+\infty}\phi_{q}\left(\dfrac{1}{\Gamma(\gamma)}\displaystyle\int_{0}^{s}(s-\tau)^{\gamma-1}a(\tau)d\tau \right)ds=\ds\int_{0}^{+\infty}\ds\int_{0}^{s}a(\tau)d\tau ds=1<+\infty.$$ Hence the condition (H3) is satisfied.
	Choose $B_{\delta}=\dfrac{3\delta+1}{9}\geq\dfrac{1}{3.921}$. It is easy see by calculating that by \eqref{eq.6}
	 $$ L=\displaystyle\frac{\Gamma(\alpha+\beta-1)}{\Gamma(\alpha)[  \Gamma(\alpha+\beta-1)-\displaystyle\sum_{i=1}^{m-2}\eta_i\xi_i^{\alpha+\beta-2}]}=\displaystyle\frac{1}{\frac{3\sqrt{\pi}}{4}(1-\frac{1}{3})}=\dfrac{2}{\sqrt{\pi}}\approx1.128$$ and by \eqref{eq.9} $$M=L(\alpha-1)\phi_{q}(B_\delta)\displaystyle\int_{0}^{+\infty}\displaystyle\int_{0}^{s}a(\tau)d\tau ds=\dfrac{2}{\sqrt{\pi}}\dfrac{3}{2}B_\delta.1=\dfrac{3}{\sqrt{\pi}}B_\delta\approx(1.693)B_\delta.$$
Finally, since $\dfrac{9B_\delta -1}{3}\geq (1,693)B_\delta$ with $B_\delta\geq\dfrac{1}{3.921}$, the condition \eqref{eq3.1} is satisfied. By means of Theorem \ref{teo3.1}, we conclude that the fractional BVP \eqref{eq4.1} has at least one positive solution $u=u(t)$ such that
	$$ 0\leq\ds\frac{u(t)}{1+t^{\frac{3}{2}}}\leq\delta, \ \  0\leq\ds\frac{u'(t)}{1+t^{\frac{3}{2}}}\leq\delta, \ \ t\in[0, +\infty).$$


\end{document}